\documentclass[oneside,11pt,reqno]{amsart}
\usepackage{amssymb,amsmath,amsthm,bbm,enumerate,mdwlist,url,multirow,hyperref,amsthm,stmaryrd}
\usepackage[pdftex]{graphicx}
\usepackage[shortlabels]{enumitem}

\theoremstyle{definition}
\newtheorem{definition}{Definition}
\theoremstyle{theorem}
\newtheorem{proposition}[definition]{Proposition}
\newtheorem{lemma}[definition]{Lemma}
\newtheorem{theorem}[definition]{Theorem}

\newtheorem{assumption}[definition]{Assumption}
\numberwithin{equation}{section}
\numberwithin{definition}{section}
\theoremstyle{remark}
\newtheorem{remark}[definition]{Remark}

\def\PP{\mathsf P}
\def\FF{\mathcal F}
\def\diffusion{\sigma^2}
\def\drift{\gamma}
\def\measure{\nu}
\def\set{V}

\def\Wq{W^{(q)}}
\def\Wl{W^{(\lambda)}}

\def\Iq{I^{(q)}}
\def\nn{\mathbf{n}}
\begin{document}
\title{Excursions of a spectrally negative L\'evy process from a two-point set}
\author{Matija Vidmar}
\address{Department of Mathematics, University of Ljubljana and Institute of Mathematics, Physics and Mechanics, Slovenia}
\email{matija.vidmar@fmf.uni-lj.si}

\begin{abstract}
Let $a\in (0,\infty)$. For a spectrally negative L\'evy process  $X$ with infinite variation paths the resolvent of the process  killed on hitting the two-point set $\set=\{-a,a\}$ is identified.  When further $X$ has no diffusion component  the Laplace transforms of the entrance laws of the excursion measures of $X$ from $V$ are determined. This is then applied to establishing the Laplace transform of the amount of time that elapses between the last visit of $X$ to a given point $x$, before hitting some other point $y>x$, and the hitting time of $y$. 
All the expressions are explicit and tractable in the standard fluctuation quantities associated to $X$.
\end{abstract}

\thanks{The author acknowledges support from the Slovenian Research Agency (research core funding No. P1-0222) and is grateful for the hospitality of CIMAT, during a visit to which the majority of the present research was conducted. 
}


\maketitle

\section{Introduction}
Let $X$ be a spectrally negative L\'evy process \cite[Chapter 8]{kyprianou}, realized canonically on the space $\Omega$ of real-valued c\`adl\`ag paths with lifetime, conforming with the usual hypotheses of the general theory of Markov processes. In particular $X$ comes equipped with a family of laws $(\PP^x)_{x\in \mathbb{R}}$ corresponding to its starting position, time-shift operators $(\theta_t)_{t\in [0,\infty)}$, (unaugmented) natural filtration $(\FF^0_t)_{t\in [0,\infty)}$, cemetery state $\partial\notin \mathbb{R}$, and associated lifetime $\zeta$. For brevity set $\PP:=\PP^0$: then $\PP(X_0=0)=0$ and, for $x\in \mathbb{R}$, $\PP^x$ is the law of $X+x$ under $\PP$. We  denote by $\psi$ the Laplace exponent of $X$: $$\PP[ e^{\lambda X_t}]=e^{\psi(\lambda) t}, \quad \psi(\lambda)=\frac{1}{2}\diffusion\lambda^2+\gamma\lambda+\int_{(-\infty,0)}\left(e^{\lambda x}-1-\lambda x\mathbbm{1}_{[-1,0)}(x)\right)\measure(dx),\quad \{\lambda,t\}\subset [0,\infty),$$ where $\measure$ is a measure on $\mathcal{B}_\mathbb{R}$, carried by $(-\infty,0)$ and satisfying $\int( 1\land x^2)\measure(dx)<\infty$, $\diffusion\in [0,\infty)$ and $\drift\in \mathbb{R}$. 

Fix now an $a\in (0,\infty)$. In the present paper the aim is to identify the laws of the excursions of $X$ from the-two point set $V:=\{-a,a\}$. Such identification is only non-trivial when the point $0$ is regular for itself, and we make the standing assumption throughout the remainder of this paper that this is so: recall that it is equivalent to assuming that 
\begin{assumption}\label{assumption}
Either $\diffusion>0$ or else $\int (1\land \vert x\vert)\Pi(dx)=\infty$.
\end{assumption}
The latter also guarantees that  $X$ exits and enters $V$ continuously (i.e. is continuous at the left and the right endpoints of the open intervals contiguous to $M:=\overline{\{t\in [0,\infty):X_t\in V\}}$) \cite{millar}; that it hits points \cite{kesten}; and that $0$ is regular for both of the halflines (i.e. for $(-\infty,0)$ and $(0,\infty)$) \cite[p. 232]{kyprianou}.

In order to make our mandate more precise we introduce now some notation concerning the `exit' measures of $X$ from $\set$ and their associated entrance laws, recalling in parallel some theory that we will need later on. 

In this regard begin by setting $\rho:=\inf\{t\in (0,\infty):X_t\in V\}$ and let $G$ be the set of the strictly positive left endpoints of the intervals contiguous to $M$. Then according to \cite[Chapter~VII]{blumenthal} \cite{maisonneuve} there exist 

(1)  a local time at $V$, i.e. a, up to indistinguishability unique, continuous additive functional $L=(L_t)_{t\in [0,\infty)}$, whose support in the sense of measures is indistinguishable from $M$, and whose $1$-potential is $\PP^\cdot [\int_0^\infty e^{-t}dL_t]=\PP^\cdot[ e^{-\rho}]$; 

and 

(2) a system of `exit' measures from $V$, i.e. a unique pair of  measures $\hat\PP^a$ and $\hat\PP^{-a}$ on $\Omega$ satisfying $\hat \PP^{\pm a}(\rho=0)=0$,  $\hat \PP^{\pm a}[1-e^{-\rho}]\leq 1$ and (the master formula) for all predictable nonnegative $Z$ and $\FF^0_\infty/\mathcal{B}_{[0,\infty]}$-measurable $f$, identically,
$$\PP^\cdot\left[\sum_{s\in G}Z_s\cdot f\circ \theta_s\right]=\PP^\cdot \left[\int_0^\infty Z_s \hat \PP^{X_s}[f]dL_s\right].$$ 
The pair $(L,\hat\PP^\cdot)$ is called the  exit system for $X$ from $V$. This system then furthermore enjoys the Markov property: whenever $T$ is an $(\FF^0_{t+})_{t\in [0,\infty)}$-stopping time with $T>0$, $g$ an $\FF^0_{T+}$/$\mathcal{B}_{[0,\infty]}$-measurable map, and $F$ an $\FF^0_\infty$/$\mathcal{B}_{[0,\infty]}$-measurable map, then $$\hat\PP^{\pm a}[g\cdot F\circ \theta_T;T<\rho]=\hat \PP^{\pm a} [g\cdot\PP^{X_T}(F);T<\rho]$$ (where one may remove the stipulation $T>0$ provided $\{T=0\}\subset \{X_0\notin V\}$). Moreover, thanks to $X$ exiting $V$ continuously and thanks to the master formula, we have $\hat \PP^{\pm a}(X_0\ne \pm a)=0$, whence (see the discussion of \cite[p. 233]{blumenthal}) the master formula persists for all optional $Z$. We recall also that by the results of \cite{millar} (see discussion in \cite[p. 85]{ppr}) when $\diffusion=0$ then $\hat \PP^{a}(\tau_a^-=0)=\hat \PP^{-a}(\tau_{-a}^-=0)=0$, while $\hat\PP^{a}(\tau_{ a}^\pm=0)$ and  $\hat\PP^{-a}(\tau_{ -a}^\pm=0)$ are all non-zero when $\diffusion>0$.

For $t\in (0,\infty)$ we next set $$\eta_t^{\pm a}(dx):=\hat\PP^{\pm a}(X_t\in dx, t<\rho)$$ and further for $x\in \mathbb{R}$, $$Q_t(x,dy):=\PP^x(X_t\in dy,t<\rho).$$ $Q_t(x,dy)$ is the semigroup of $X$ killed on hitting $V$ and $(\eta_t^{\pm a})_{t\in (0,\infty)}$ are entrance laws for $(Q_t)_{t\in (0,\infty)}$, in the sense that, for $\{t,s\}\subset (0,\infty)$, $\eta_{t+s}^{\pm a}=\eta_t^{\pm a}Q_s$: indeed, for $\mathcal{B}_\mathbb{R}/\mathcal{B}_{[0,\infty]}$-measurable $f$, by the Markov property under $\hat \PP^{\pm a}$,  $\eta_{t+s}^{\pm a}[f]=\hat \PP^{\pm a}[f(X_{t+s});t+s<\rho]=\hat \PP^{\pm a}[\PP^{X_t}[f(X_s);s<\rho];t<\rho]=\eta_t^{\pm a} [Q_s(\cdot,f)]$. 

Lastly, for $q\in (0,\infty)$, $x\in \mathbb{R}$, we let the resolvent $V_q(x,\cdot)$  of the killed process be defined via $$V_q(x,dy):=\int_0^\infty e^{-q t}Q_t(x,dy)dt,$$ and we set $$\hat \eta^{\pm a}_q:=\int_0^\infty e^{-qt}\eta_t^{\pm a}dt.$$

We provide information about the structure of the excursions from $\set$ by identifying   (i) $(V_q)_{q\in (0,\infty)}$ (Proposition~\ref{proposition:killed}) and (ii)  $(\hat\eta_q^{- a})_{q\in (0,\infty)}$ and, when $\diffusion=0$, $(\hat\eta_q^{ a})_{q\in (0,\infty)}$ (Theorem~\ref{theorem:main}; see Remark~\ref{remark:missing} for the reasons behind the omission of $(\hat\eta_q^{ a})_{q\in (0,\infty)}$  in the case when $ \diffusion>0$), which is related to analyzing the process conditioned to avoid $\set$ (Theorem~\ref{theorem:conditioned}). The expressions obtained are explicit in the usual scale functions (see Section~\ref{section:notation}) associated to $X$. As an application we characterize the law of the amount of time that elapses between the last visit of $X$ to a given point $x$, before hitting some other point $y>x$, and the hitting time of $y$, on the event that $y$ is hit at all (Theorem~\ref{theorem:application}). 

This investigation may be seen as a natural offspring of the results of \cite{ppr}, wherein the excursion measure of $X$ away from a point was studied, and to which we refer for a  review of the relevant literature.

The organization of the remainder of this paper is as follows. Section~\ref{section:notation} introduces some further standard notation from the fluctuation theory for $X$. In Section~\ref{section:local}  we establish connections between $L$ and $\hat{\PP}^{\pm a}$ on the one hand and the apposite quantities that are to do with visits to a single point on the other. Then Section~\ref{section:killed/conditioned-to-avoid} discusses the process $X$ killed on hitting $\set$ and conditioned to avoid $\set$. The latter is brought to bear on the problem of determining $(\hat\eta_q^{\pm a})_{q\in (0,\infty)}$ in Section~\ref{section:entrance-laws}, before closing with an application in  Section~\ref{section:applications}.

\section{Review of some further notation from the fluctuation theory of $X$}\label{section:notation}
The Laplace exponent $\psi$ has a largest zero, which we denote by $\Phi(0)$. We set $\Phi:=(\psi\vert_{[\Phi(0),\infty)})^{-1}$ for the right-continuous inverse of $\psi$. From $\psi\circ \Phi=\mathrm{id}_{[0,\infty)}$, it follows upon differentiation that $(\psi'\circ \Phi)\Phi'=1$, at least on $(0,\infty)$. 

Associated to $X$ is a family of scale functions $(\Wq)_{q\in [0,\infty)}$ that feature heavily in first passage/exit and related fluctuation identities \cite[Chapter~8]{kyprianou}. Specifically, for $q\in [0,\infty)$, $\Wq$ is characterized as the unique function mapping $\mathbb{R}$ to $[0,\infty)$, vanishing on $(-\infty,0)$, continuous on $[0,\infty)$ and having Laplace transform $$\int_0^\infty e^{-\theta x}\Wq(x)dx=\frac{1}{\psi(\theta)-q},\quad \theta\in (\Phi(q),\infty).$$ Since we are assuming $X$ has unbounded variation, $\Wq(0)=0$  \cite[Lemma~8.6]{kyprianou} and $\Wq$ is of class $C^1$ on $(0,\infty)$ \cite[p. 241]{kyprianou} for all $q\in [0,\infty)$. 

For $q\in (0,\infty)$, we denote by $U_q$ the $q$-resolvent measure of $X$, $U_q(dy):=\int_0^\infty e^{-qt}\PP(X_t\in dy)dt$. It admits a continuous density $u_q$ with respect to Lebesgue measure: $u_q(y)=\Phi'(q)e^{-\Phi(q)y}-\Wq(-y)$ for $y\in \mathbb{R}$ \cite[Corollary~8.9]{kyprianou}. 

Finally, for $z\in \mathbb{R}$, we set $\tau^+_z:=\inf\{t\in (0,\infty):X_t>z\}$, $\tau^-_z:=\inf\{t\in (0,\infty):X_t<z\}$ and  $T_z:=\inf\{t\in (0,\infty):X_t=z\}$.

\section{Local time at $V$ and connection to exit measure from zero of \cite{ppr}}\label{section:local}
For $x\in \mathbb{R}$, let $\tilde{\PP}^x$ be the exit measure from the point $x$ and let $L^x$ be the local time at $x$ of $X$, normalized in such a way that, with $G^x$ being the set of strictly positive left endpoints of the open intervals contiguous to $\overline{\{t\in [0,\infty):X_t=x\}}$, (the master formula) $\PP^\cdot[\sum_{s\in G^{x}}Z_s f\circ \theta_s]=\PP^\cdot[ \int_0^\infty Z_s dL_s^{a}]\tilde{\PP}^{x}[f]$ holds true for all predictable nonnegative $Z$ and $\FF^0_\infty/\mathcal{B}_{[0,\infty]}$-measurable $f$, while $\PP^\cdot[ e^{-T_x}]=\PP^\cdot[\int_0^\infty e^{-s}dL^x_s]$.

Set, for $q\in (0,\infty)$,
\begin{equation}\label{eq:local-times}
\alpha_a(q):=e^{\Phi(q)2a}\left(1-\frac{\Phi'(q)}{\Wq(2a)}(e^{\Phi(q)2a}-1)\right)\text { and }\alpha_{-a}(q):=\frac{\Phi'(q)}{\Wq(2a)}(e^{\Phi(q)2a}-1).
\end{equation}

\begin{remark}
For $q\in (0,\infty)$, by optional sampling of the exponential martingale $(e^{\Phi(q)X_t-qt})_{t\in [0,\infty)}$, $\PP[ e^{-q T_{2a}}]=e^{-\Phi(q)2a}$ \cite[Eq.~(3.15)]{kyprianou}, whilst $\PP [e^{-qT_{-2a}}]=\PP^{2a}[e^{-qT_0}]=\frac{u_q(-2a)}{u_q(0)}=e^{\Phi(q)2a}-\Wq(2a)/\Phi'(q)$ \cite[Eq.~(9)]{ppr}. Hence $\alpha_a(q)=\frac{1-\PP [e^{-qT_{-2a}}]}{1-\PP[ e^{-qT_{-2a}}]\PP [e^{-qT_{2a}}]}$ and $\alpha_{-a}(q)=\frac{1-\PP[ e^{-qT_{2a}}]}{1-\PP[ e^{-qT_{-2a}}]\PP[ e^{-qT_{2a}}]}$.
\end{remark}

\begin{proposition}
$L=\alpha_a(1)L^a+\alpha_{-a}(1)L^{-a}$ up to indistinguishability.
\end{proposition}

\begin{proof}
$\alpha_a(1)L^a+\alpha_{-a}(1)L^{-a}$ is a continuous additive functional whose support is $M$. Moreover, $\PP^\cdot [\int_0^\infty e^{-s}d(\alpha_a(1)L^a+\alpha_{-a}(1)L^{-a})_s]=\PP^\cdot [e^{-\rho}\PP^{X_\rho}[\alpha_a(1) \int_0^\infty e^{-s}dL^a_s+\alpha_{-a}(1)\int_0^\infty e^{-s}dL^{-a}_s]]=\alpha_a(1)\PP^\cdot [e^{-\rho};X_\rho=a]+\alpha_a(1)\PP^\cdot[ e^{-\rho};X_\rho=-a]\PP^{-a}[e^{-T_a}]+\alpha_{-a}(1)\PP^{\cdot}[e^{-\rho};X_\rho=a]\PP^a [e^{-T_a}]+\alpha_{-a}(1)\PP^\cdot [e^{-\rho};X_\rho=-a]=\PP^\cdot [e^{-\rho};X_\rho=a](\alpha_a(1)+\alpha_{-a}(1)\PP [e^{-T_{-2a}}])+\PP^\cdot [e^{-\rho};X_\rho=-a](\alpha_a(1)\PP[e^{-T_{2a}}]+\alpha_{-a}(1))=\PP^\cdot [e^{-\rho}]$. 
\end{proof}

Next, if in the master formulae for exits from $a$ and $\{-a,a\}$ we take $f=g$ and $f=g\mathbbm{1}(X_0=a)$, respectively, for an $\FF^0_\infty/\mathcal{B}_{[0,\infty]}$-measurable $g$, then we see (from the suitable uniqueness of the exit systems) that  $\alpha_a(1)\hat\PP^a=\tilde{\PP}^a$. Similarly $\alpha_{-a}(1)\hat\PP^{-a}=\tilde{\PP}^{-a}$. Moreover, by spatial homogeneity, $\tilde{\PP}^{\pm a}$ is the push-forward of $\tilde{\PP}^0$ under translation by $\pm a$. Finally, it is easy to relate $\tilde{\PP}^0$ to the exit measure away from zero as normalized in \cite{ppr}, and where it was denoted $\nn$. Indeed, since (from the master formula) $\tilde{\PP}^0[1-e^{-T_0}]=1$, it follows from \cite[Eq.~(17)]{ppr} that $\tilde{\PP}^0=u_1(0)\nn=\Phi'(1)\nn$.

\section{The process $X$ killed on hitting, and conditioned to avoid $\set$}\label{section:killed/conditioned-to-avoid}
Let $e_1$ be a mean one exponentially distributed random time independent of $X$ (under each $\PP^\cdot$ and each $\hat{\PP}^\cdot$, by innocuously enlarging the underlying space).  For $q\in (0,\infty)\backslash \{1\}$, set $e_q:=e_1/q$, and then further for $x\in \mathbb{R}$, $$h_q(x):=\PP^x(\rho>e_q).$$  The device of using the independent exponential random times  will be for convenience/matters of interpretation (e.g. ``killing at an independent exponential random time'') only: in particular in all the expressions one can always just integrate them out to arrive at statements/expressions that omit their usage. For instance $h_q(x)$ is the probability that $X$ started from $x$ and killed at $e_q$ avoids $\set$, but it could just as well be written $\PP^x[1-e^{-q\rho}]$.

The following proposition, which is our first main result, identifies explicitly the resolvent of the process $X$ killed on hitting $\set$. 

\begin{proposition}\label{proposition:killed}
Let $q\in (0,\infty)$ and $x\in \mathbb{R}$. $V_q(x,\cdot)$ is absolutely continuous with respect to Lebesgue measure and its density, denoted $v_q(x,\cdot)$, is given as follows (Lebesgue-a.e. in $y\in \mathbb{R}$): \footnotesize
$$v_q(x,y)=\Wq(x-a)e^{\Phi(q)2a}\left[e^{\Phi(q)(-a-y)}-\frac{\Wq(a-y)}{\Wq(2a)}-\Wq(-a-y)\left(\Phi'(q)^{-1}-\frac{e^{\Phi(q)2a}}{\Wq(2a)}\right)\right]+$$
$$\frac{\Wq(x+a)}{\Wq(2a)}\left[\Wq(a-y)-\Wq(-a-y)e^{\Phi(q)2a}\right]+\Wq(-a-y)e^{\Phi(q)(x+a)}-\Wq(x-y);$$\normalsize
in particular for $x\in (-\infty,a]$, \footnotesize 
$$v_q(x,y)=\frac{\Wq(x+a)}{\Wq(2a)}\left[\Wq(a-y)-\Wq(-a-y)e^{\Phi(q)2a}\right]+\Wq(-a-y)e^{\Phi(q)(x+a)}-\Wq(x-y);$$\normalsize
and if even $x\in (-\infty,-a]$, then $$v_q(x,y)=\Wq(-a-y)e^{\Phi(q)(x+a)}-\Wq(x-y).$$

Furthermore, \footnotesize$$h_q(x)=1-e^{\Phi(q)(a+x)}+\Wq(x+a)\frac{e^{\Phi(q)2a}-1}{\Wq(2a)}+\Wq(x-a)e^{\Phi(q)2a}\left(\Phi'(q)^{-1}-\frac{e^{\Phi(q)2a}-1}{\Wq(2a)}\right);$$ \normalsize in particular when $x\in (-\infty,a]$, $h_q(x)=1-e^{\Phi(q)(a+x)}+\Wq(x+a)\frac{e^{\Phi(q)2a}-1}{\Wq(2a)}$; and if even $x\in (-\infty, -a]$, then $h_q(x)=1-e^{\Phi(q)(a+x)}$. 
\end{proposition}
\begin{proof}
Some preliminary observations. 

First, whenever $\{c,b\}\subset \mathbb{R}$, $c< b$, $z\in [c,b]$, then the $q$-potential measure of $X$ killed on exiting $[c,b]$, when starting from $z$, admits a density $u^{(q)}(c,b,z,y)$ (in $y\in \mathbb{R}$, with respect to Lebesgue measure), which is given by 
\begin{equation}\label{eq:help:zero-o}
u^{(q)}(c,b,z,y)=\mathbbm{1}_{[c,b]}(y)\left(\frac{\Wq(z-c)\Wq(b-y)}{\Wq(b-c)}-\Wq(z-y)\right)
\end{equation}
 \cite[Theorem~8.7]{kyprianou}. As a consequence, by monotone convergence and the asymptotic behavior of $\Wq$ at infinity  \cite[Lemma~3.3]{kkr}, letting $b\uparrow \infty$, respectively $c\downarrow -\infty$, the $q$-potential measure of $X$ killed on exiting $[c,\infty)$, respectively hitting $b$, when starting from $z\in [c,\infty)$, respectively $z\in (-\infty,b]$, is absolutely continuous with respect to Lebesgue measure, and has the density 
\begin{equation}\label{eq:help:zero-i}
\mathbb{R}\ni y\mapsto \mathbbm{1}_{[c,\infty)}(y)\left(\Wq(z-c)e^{\Phi(q)(c-y)}-\Wq(z-y)\right),
\end{equation} respectively 
\begin{equation}\label{eq:help:zero-ii}
\mathbb{R}\ni y\mapsto \mathbbm{1}_{(-\infty,b]}(y)\left(\Wq(b-y)e^{\Phi(q)(z-b)}-\Wq(z-y)\right).
\end{equation} 

Second, for any $t\in (0,\infty)$ and $\{c,d\}\subset \mathbb{R}$, 
\begin{equation}\label{eq:help:minor}
\PP^d(\tau^-_c=t)=0
\end{equation}
 (for instance since $\{\tau^-_c=t\}\subset \{X\text{ jumps at }t\}\cup \{X_t=c\}$, the first event being negligible by the stochastic continuity of L\'evy processes, the latter by \cite[Theorem~27.4]{sato} under our standing Assumption~\ref{assumption}).

Third, as follows at once by optional sampling of the martingale $(e^{-qt+\Phi(q)X_t})_{t\in [0,\infty)}$ and from \cite[Eq.~(8.11)]{kyprianou}, for all $b\in (0,\infty)$ and $z\in (-\infty,b]$, 
\begin{equation}\label{eq:help:one:i}
\PP^z\left[e^{-q \tau_0^-+\Phi(q) X_{\tau_0^-}};\tau_0^-<\tau_b^+\right]=e^{\Phi(q)z}-\frac{\Wq(z)}{\Wq(b)}e^{\Phi(q)b}.
\end{equation} 
In particular, letting $b\uparrow \infty$, using again the asymptotics of $\Wq$ \cite[Lemma~3.3]{kkr}, for all $z\in \mathbb{R}$,
\begin{equation}\label{eq:help:one:ii}
\PP^z\left[e^{-q \tau_0^-+\Phi(q) X_{\tau_0^-}};\tau_0^-<\infty\right]=e^{\Phi(q)z}-\Wq(z)\Phi'(q)^{-1}.
\end{equation} 

Finally, we observe that, thanks to $(\Wq(X)e^{-q\cdot})^{\tau_0^-\land \tau^+_b}$ being a bounded martingale for each $b\in (0,\infty)$ \cite[Exercise~8.12]{kyprianou}, so, by optional stopping, is $(\Wq(X)e^{-q\cdot})^{\tau_c^-\land \tau^+_b}$ for any $c\in [0,\infty)$. In particular, for any $z\in  (-\infty,b]$, $\Wq(z)=\PP^z[\Wq(X_{\tau_c^-})e^{-q\tau_c^-};\tau_c^-<\tau_b^+]+\Wq(b)\PP^z[e^{-q\tau_b^+};\tau_b^+<\tau_c^-]$; hence  by \cite[Eq.~(8.11)]{kyprianou}, for $c<b$,
\begin{equation}\label{eq:help:two:i}
\PP^z[\Wq(X_{\tau_c^-})e^{-q\tau_c^-};\tau_c^-<\tau_b^+]=\Wq(z)-\frac{\Wq(z-c)}{\Wq(b-c)}\Wq(b).
\end{equation}
In particular, passing to the limit $b\uparrow\infty$, using yet again the asymptotic behavior of $\Wq$  \cite[Lemma~3.3]{kkr}, allows to conclude that for $z\in \mathbb{R}$, 
\begin{equation}\label{eq:help:two:ii}
\PP^z[\Wq(X_{\tau_c^-})e^{-q\tau_c^-};\tau_c^-<\infty]=\Wq(z)-\Wq(z-c)e^{\Phi(q)c}.
\end{equation}
Take now an arbitrary $\mathcal{B}_\mathbb{R}/\mathcal{B}_{[0,\infty]}$-measurable $f$. 

Suppose first $x\in (-\infty,-a]$. Then for any $t\in (0,\infty)$, $\PP^x[f(X_t);t<\rho]=\PP^x[f(X_t);t<\tau^+_{-a}]$, and the claim follows upon integrating against $\mathbbm{1}_{(0,\infty)}(t)e^{-qt}dt$ and taking into account \eqref{eq:help:zero-ii}. 

Suppose now $x\in [-a,a]$. Then for any $t\in (0,\infty)$, $\PP^x[f(X_t);t<\rho]=\PP^x[f(X_t);t<\tau^+_a\land T_{-a}]=\PP^x[f(X_t);t<\tau_{-a}^-\land\tau_a^+]+\PP^x[f(X_t);\tau^-_{-a}\leq t<\tau_a^+\land T_{-a}]=\PP^x[f(X_t);t<\tau_{-a}^-\land\tau_a^+]+\PP^x\left[\PP^{X_{\tau^-_{-a}}}\left[f(X_{t-s});t-s<\tau^+_{-a}\right]\vert_{s=\tau^-_{-a}};\tau^-_{-a}<t\land \tau_a^+\right]$, where we have used \eqref{eq:help:minor} in the third equality. Now integrate against $\mathbbm{1}_{(0,\infty)}(t)e^{-qt}dt$. One obtains that $v_q(x,y)$ is equal to the sum of $u^{(q)}(-a,a,x,y)$ from \eqref{eq:help:zero-o} and of the relevant kernel that will appear in 
$$\int_0^\infty e^{-qt}\PP^x\left[\PP^{X_{\tau^-_{-a}}}\left[f(X_{t-s});t-s<\tau^+_{-a}\right]\big\vert_{s=\tau^-_{-a}};\tau^-_{-a}<t\land \tau_a^+\right]dt$$
$$=\PP^x\left[\int_0^\infty e^{-qt}\PP^{X_{\tau^-_{-a}}}\left[f(X_{t-s});t-s<\tau^+_{-a}\right]\mathbbm{1}(s<t\land \tau_a^+)dt\big \vert_{s=\tau^-_{-a}}\right]$$
$$=\PP^x\left[e^{-qs}\int_s^\infty e^{-q(t-s)}\PP^{X_{\tau^-_{-a}}}\left[f(X_{t-s});t-s<\tau^+_{-a}\right]dt\mathbbm{1}(s<\tau_a^+)\vert_{s=\tau^-_{-a}}\right]$$
$$=\PP^x\left[\int_0^\infty e^{-qt}\PP^{X_{\tau^-_{-a}}}\left[f(X_{t});t<\tau^+_{-a}\right]dte^{-q\tau^-_{-a}};\tau_{-a}^-<\tau_a^+\right]$$
$$=\PP^x\left[\int_{-\infty}^{-a}\left(\Wq(-a-y)e^{\Phi(q)(X_{\tau^-_{-a}}+a)}-\Wq(X_{\tau^-_{-a}}-y)\right)dye^{-q\tau^-_{-a}};\tau_{-a}^-<\tau_a^+\right].$$ It remains to identify, for $y\in (-\infty,-a)$, $$\PP^x\left[\left(\Wq(-a-y)e^{\Phi(q)(X_{\tau^-_{-a}}+a)}-\Wq(X_{\tau^-_{-a}}-y)\right)e^{-q\tau^-_{-a}};\tau_{-a}^-<\tau_a^+\right].$$
But, by \eqref{eq:help:one:i}, $\PP^x\left[e^{\Phi(q)(X_{\tau^-_{-a}}+a)-q\tau^-_{-a}};\tau_{-a}^-<\tau_a^+\right]=\PP^{x+a}\left[e^{\Phi(q)(X_{\tau^-_0})-q\tau^-_0};\tau_0^-<\tau_{2a}^+\right]=e^{\Phi(q)(x+a)}-\frac{\Wq(x+a)}{\Wq(2a)}e^{\Phi(q)2a}$, whilst by \eqref{eq:help:two:i}, $\PP^x\left[\Wq(X_{\tau^-_{-a}}-y)e^{-q\tau^-_{-a}};\tau_{-a}^-<\tau_a^+\right]=\PP^{x-y}\left(\Wq(X_{\tau^-_{-a-y}})e^{-q\tau^-_{-a-y}};\tau_{-a-y}^-<\tau_{a-y}^+\right)=\Wq(x-y)-\frac{\Wq(x+a)}{\Wq(2a)}\Wq(a-y)$. The desired expression follows.

Finally suppose $x\in [a,\infty)$.  Then for any $t\in (0,\infty)$, $\PP^x[f(X_t);t<\rho]=\PP^x[f(X_t);t<\tau_a^-]+\PP^x[f(X_t);\tau_a^-\leq t<\rho]=\PP^x[f(X_t);t<\tau_a^-]+\PP^x\left[\PP^{X_{\tau^-_a}}\left[f(X_{t-s});t-s<\rho\right]\vert_{s=\tau_a^-};\tau_a^-<t\right].$ Integrating against $\mathbbm{1}_{(0,\infty)}(t)e^{-qt}dt$ we find that $v_q(x,y)$ will be the sum of an expression given by \eqref{eq:help:zero-i} and of the relevant kernel that will appear in 

$$\int_0^\infty e^{-qt}\PP^x\left[\PP^{X_{\tau^-_a}}\left[f(X_{t-s});t-s<\rho\right]\vert_{s=\tau_a^-};\tau_a^-<t\right]dt$$
$$=\PP^x\left[\int_0^\infty e^{-qt}\PP^{X_{\tau^-_a}}\left[f(X_{t-s});t-s<\rho\right]\mathbbm{1}_{(s,\infty)}(t)dt\vert_{s=\tau_a^-}\right]$$
$$=\PP^x\left[e^{-q\tau_a^-}\int_0^\infty e^{-qt}\PP^{X_{\tau^-_a}}\left[f(X_{t});t<\rho\right]dt;\tau_a^-<\infty\right]$$
$$=\PP^x\Bigg[e^{-q\tau_a^-}\int_{-\infty}^\infty f(y)\Bigg[\mathbbm{1}_{(-\infty,a)}(y)\left(\frac{\Wq(X_{\tau^-_a}+a)\Wq(a-y)}{\Wq(2a)}-\Wq(X_{\tau^-_a}-y)\right)$$
$$+\mathbbm{1}_{(-\infty,-a)}(y)\Wq(-a-y)\left(e^{\Phi(q)(X_{\tau^-_a}+a)}-\frac{\Wq(X_{\tau^-_a}+a)}{\Wq(2a)}e^{\Phi(q)2a}\right)\Bigg]dy;\tau_a^-<\infty\Bigg].$$

It remains to identify, as follows. By \eqref{eq:help:two:ii}: $\PP^x[\Wq(X_{\tau^-_a}+a)e^{-q\tau_a^-};\tau_a^-<\infty]=\PP^{x+a}[\Wq(X_{\tau^-_{2a}})e^{-q\tau_{2a}^-};\tau_{2a}^-<\infty]=\Wq(x+a)-\Wq(x-a)e^{\Phi(q)2a}$; for $y\in (-\infty,a)$, also by \eqref{eq:help:two:ii}: $\PP^x[\Wq(X_{\tau^-_a}-y)e^{-q\tau_a^-};\tau_a^-<\infty]=\PP^{x-y}[\Wq(X_{\tau^-_{a-y}})e^{-q\tau^-_{a-y}};\tau_{a-y}^-<\infty]=\Wq(x-y)-\Wq(x-a)e^{\Phi(q)(a-y)}$; by  \eqref{eq:help:one:ii}: $\PP^x\left[e^{\Phi(q)(X_{\tau^-_a}+a)-q\tau^-_a};\tau_a^-<\infty\right]=\PP^{x-a}\left[e^{(X_{\tau^-_0}+2a)\Phi(q)-q\tau^-_0};\tau_0^-<\infty\right]=e^{\Phi(q)(x+a)}-\Wq(x-a)\Phi'(q)^{-1}e^{\Phi(q)2a}$.

Next, the second part of the proposition could  be got by integrating the resolvent densities of the killed process just obtained (because $qV_q(x,\mathbbm{1})=\PP^x(e_q<\rho)$), but it seems easier to do the computations directly. Indeed, we have $h_q(x)=\PP^x(\rho>e_q)=\PP^x[1-e^{-q \rho}]$. Then we compute, as follows. (Since the methods are similar to the ones employed in obtaining the resolvent densities, we omit making explicit some of the details.)

For $x\in (-\infty,-a]$, $\PP^x [e^{-q \rho}]=\PP^x [e^{-q \tau_{-a}^+}]=e^{-\Phi(q)(-a-x)}$.   

Then, for $x\in [-a,a]$, \footnotesize $$\PP^x[ e^{-q \rho}]=\PP^x[e^{-q \tau_a^+};\tau_a^+<\tau_{-a}^-]+\PP^x[e^{-q \rho};\tau_{-a}^-<\tau_a^+]=\frac{\Wq(a+x)}{\Wq(2a)}+\PP^x[\PP_{X_{\tau_{-a}^-}}[e^{-q \tau_{-a}^+}]e^{-q \tau_{-a}^-};\tau_{-a}^-<\tau_a^+].$$\normalsize
But the second term in the preceding sum is equal to $\PP^x\left[e^{-\Phi(q)(-a-X_{\tau_{-a}^-})-q\tau_{-a}^-};\tau_{-a}^-<\tau_a^+\right]=e^{\Phi(q)a}\left(e^{x\Phi(q)}-\frac{\Wq(x+a)}{\Wq(2a)}e^{\Phi(q)a}\right)$.

Finally, for $x\in [a,\infty)$, 
$$\PP^x[e^{-q \rho}]=\PP^x[e^{-q \tau_a^-}(e^{-q \rho})\circ\theta_{\tau_a^-}; \tau_a^-<\infty]=\PP^x[e^{-q \tau_a^-}\PP_{X_{\tau_a^-}}[e^{-q \rho}];\tau_a^-<\infty]$$
$$=\PP^x\left[e^{-q \tau_a^-}\left[\frac{\Wq(a+X_{\tau_a^-})}{\Wq(2a)}+e^{\Phi(q)a}\left(e^{\Phi(q)X_{\tau_a^-}}-\frac{\Wq(X_{\tau_a^-}+a)}{\Wq(2a)}e^{\Phi(q)a}\right)\right];\tau_a^-<\infty\right]$$
$$=e^{\Phi(q)2a}\left(e^{\Phi(q)(x-a)}-\Wq(x-a)\Phi'(q)^{-1}\right)+\frac{1-e^{\Phi(q)2a}}{\Wq(2a)}\left(\Wq(x+a)-\Wq(x-a)e^{\Phi(q)2a}\right).$$
\end{proof}

\begin{remark}
More generally, if one has $N\in \mathbb{N}$ and real numbers $a_1<\cdots<a_{N+1}$, and if one denotes by $u_N^{(q)}(x,\cdot)$ the $q$-resolvent density of the process killed on hitting $\{a_1,\ldots,a_N\}$, and likewise  by $u_{N+1}^{(q)}(x,\cdot)$ the $q$-resolvent density of the process killed on hitting $\{a_1,\ldots,a_N,a_{N+1}\}$, then proceeding as in the previous proof, one has, for instance for $x\in (a_N ,a_{N+1}]$, $$u_{N+1}^{(q)}(x,y)=u^{(q)}(a_N,a_{N+1},x,y)+\PP^x[e^{-q\tau_{a_N}^-}u^{(q)}_N(X_{\tau_{a_N}^-}, y);\tau_{a_N}^-<\tau_{a_{N+1}}^+]$$ Lebesgue-a.e. in $y\in \mathbb{R}$. (of course $u_{N+1}^{(q)}(x,\cdot)=u_N^{(q)}(x,\cdot)$ for $x\in (-\infty,a_N]$; and the relevant recursion for $x\in [a_{N+1},\infty)$ would again follow similarly as in the proof above). However these expressions grow considerably in complexity with increasing $N$.
\end{remark}
Note now that, for $q\in (0,\infty)$ and $y\in \mathbb{R}\backslash V$, one has $\PP^y(e_q<\rho)=h_q(y)>0$. Thus, for each $q\in (0,\infty)$, we may define the Markov family of probabilities $(\PP_{\times,q}^y)_{y\in \mathbb{R}\backslash \set}$ on $\FF_\infty^0$ uniquely by insisting that $$\PP_{\times,q}^y(A\cap \{t<\zeta\}):=\frac{\PP^y(A \cap \{t<e_q< \rho\})}{\PP^y(e_q<\rho)},\quad A\in \FF^0_t,t\in [0,\infty),y\in \mathbb{R}\backslash \set.$$
This corresponds to $X$ killed at $e_q$ and conditioned not to hit $\set$.
\begin{remark}\label{remark:doob}
For $t\in [0,\infty)$ and $\FF_t^0/\mathcal{B}_{[0,\infty]}$-measurable $F$, by the Markov property of $X$, $\PP^y_{\times,q}[F;t<\zeta]=\frac{\PP^y[F;t<e_q<\rho]}{\PP^y(\rho>e_q)}=\frac{\PP^y[Fe^{-qt}(1-e^{-q\rho})\circ \theta_t;t<\rho]}{\PP^y(\rho>e_q)}=e^{-qt}\frac{\PP^y[F\PP^{X_t}(e_q<\rho);t<\rho]}{h_q(y)}=\frac{\PP^y[Fh_q(X_t);t<\rho]}{h_q(y)}e^{-q t}=\frac{\PP^y[Fh_q(X_t);t<\rho\land e_q]}{h_q(y)}$. Thus the probabilities $\PP^y_{\times,q}$ may also be seen as having been got by a Doob h-transform via the excessive function $h_q$.
\end{remark}

The next result identifies the resolvent of the process $X$, killed at, and conditioned to avoid $\set$ up to an independent exponential random time.

\begin{theorem}\label{theorem:conditioned}
Let $q\in (0,\infty)$, $\beta\in (0,\infty)$, $g:\mathbb{R}\to \mathbb{R}$ continuous and compactly supported in $\mathbb{R}\backslash \set$. Set $Z_{q,\beta}(x,g):=\int_0^\infty e^{-\beta t}\PP^x_{\times,q}[g(X_t);t<\zeta]dt$ for $x\in \mathbb{R}\backslash \set$. One has the convergence $Z_{q,\beta}(x,g)\to \int_{-\infty}^\infty h_q(y)z_{q,\beta}(y)g(y)dy$ as $x\uparrow\downarrow \pm a$, with $z_{q,\beta}(y)$ given as follows in the respective instances. (It is part of the statement that the denominators are not zero.)
\begin{enumerate}[(i)]
\item\label{conditioned:i} As $x\uparrow -a$: 
$z_{q,\beta}(y)=\frac{(W^{(\beta+q)})'(-a-y)-\Phi(\beta+q)W^{(\beta+q)}(-a-y)}{\Phi(q)}$.
\item\label{conditioned:ii} As $x\downarrow -a$: 
$z_{q,\beta}(y)=\frac{\frac{W^{(\beta+q)}(a-y)-W^{(\beta+q)}(-a-y)e^{\Phi(q)2a}}{ W^{(\beta+q)}(2a)}-\frac{\diffusion}{2}\left((W^{(\beta+q)})'(-a-y)-\Phi(\beta+q)W^{(\beta+q)}(-a-y)\right)}{\frac{e^{\Phi(q)2a}-1}{\Wq(2a)}-\frac{\diffusion}{2}\Phi(q)}$.
\item\label{conditioned:iii} As $x\uparrow a$: 
\footnotesize
$$z_{q,\beta}(y)=\frac{(W^{(\beta+q)})'(a-y)-W^{(\beta+q)}(a-y)\frac{(W^{(\beta+q)})'(2a)}{(W^{(\beta+q)})(2a)}+W^{(\beta+q)}(-a-y)e^{\Phi(\beta+q)2a}(\frac{(W^{(\beta+q)})'(2a)}{W^{(\beta+q)}(2a)}-\Phi(\beta+q))}{\Phi(q)e^{\Phi(q)2a}-\frac{e^{\Phi(q)2a}-1}{\Wq(2a)}(\Wq)'(2a)}.$$ 
\normalsize
\item\label{conditioned:iv} As $x\downarrow a$: 
\footnotesize
$$z_{q,\beta}(y)=\frac{e^{\Phi(\beta+q)(a-y)}-\frac{W^{(\beta+q)}(a-y)}{W^{(\beta+q)}(2a)}e^{\Phi(\beta+q)2a}-W^{(\beta+q)}(-a-y)e^{\Phi(\beta+q)2a}\left(\Phi'(\beta+q)^{-1}-\frac{e^{\Phi(\beta+q)2a}}{W^{(\beta+q)}(2a)}\right)}{e^{\Phi(q)2a}\left(\Phi'(q)^{-1}-\frac{e^{\Phi(q)2a}-1}{\Wq(2a)}\right)-\frac{\diffusion}{2}\left(\Phi(q)e^{\Phi(q)2a}-\frac{e^{\Phi(q)2a}-1}{\Wq(2a)}(\Wq)'(2a)\right)}$$
$$-\frac{\diffusion}{2}\frac{(W^{(\beta+q)})'(a-y)-W^{(\beta+q)}(a-y)\frac{(W^{(\beta+q)})'(2a)}{W^{(\beta+q)}(2a)}-W^{(\beta+q)}(-a-y)e^{\Phi(\beta+q)2a}\left(\frac{(W^{(\beta+q)})'(2a)}{W^{(\beta+q)}(2a)}+\Phi(\beta+q)\right)}{e^{\Phi(q)2a}\left(\Phi'(q)^{-1}-\frac{e^{\Phi(q)2a}-1}{\Wq(2a)}\right)-\frac{\diffusion}{2}\left(\Phi(q)e^{\Phi(q)2a}-\frac{e^{\Phi(q)2a}-1}{\Wq(2a)}(\Wq)'(2a)\right)}.$$
\normalsize
\end{enumerate}
\begin{definition}\label{definition}
The limits from Theorem~\ref{theorem:conditioned}, as $x\downarrow a$, $x\uparrow a$, $x\downarrow -a$ and $x\uparrow -a$, of $Z_{q,\beta}(x,g)$, are denoted $Z_{q,\beta}^+(a,g)$, $Z_{g,\beta}^-(a,g)$, $Z_{q,\beta}^+(-a,g)$ and $Z_{g,\beta}^-(-a,g)$, respectively. 
\end{definition}
\end{theorem}
\begin{remark}\label{remark:conditioned}
The value of the denominator in the cases \ref{conditioned:ii}-\ref{conditioned:iv} (respectively, \ref{conditioned:i}-\ref{conditioned:iii}) is equal to $H(q):=\lim_{x\downarrow \pm a} h_q(x)/\Wq(x\mp a)$  (respectively, $H(q):=-\lim_{x\uparrow \pm a} h_q(x)/(x\mp a)$) --- see the proof. The map $(0,\infty)\ni q\mapsto h_q(x)$ being nondecreasing, and $\lim_{z\downarrow 0}\frac{W^{(q_1)}(z)}{W^{(q_2)}(z)}=1$ whenever $\{q_1,q_2\}\subset [0,\infty)$, it follows that the map $(0,\infty)\ni q\mapsto H(q)$ is nondecreasing. 
\end{remark}
\begin{proof}
In all cases, by Remark~\ref{remark:doob}, $Z_{q,\beta}(x,g)=\frac{V_{\beta+q}(x,gh_q)}{h_q(x)}=\int_{-\infty}^\infty v_{\beta+q}(x,y)g(y)h_q(y)dy/h_q(x)$. Now apply Proposition~\ref{proposition:killed}, so that, for instance for $x\in (-\infty,-a)$, one has $Z_{q,\beta}(x,g)=\int_{-\infty}^{\infty}g(y)(1-e^{\Phi(q)(a+y)})\frac{\left(W^{(\beta+q)}(-a-y)e^{\Phi(\beta+q)(x+a)}-W^{(\beta+q)}(x-y)\right)}{1-e^{\Phi(q)(a+x)}}dy$ with similar (but more involved) expressions in the remaining instances. 

It remains to pass to the relevant limits via  bounded convergence, using the assumption that $g$ is compactly supported in $\mathbb{R}\backslash V$, and the following facts: (a) the scale functions are $C^1$ on $\mathbb{R}\backslash \{0\}$, (b) $(\Wq)'(0+)=2/\diffusion$ ($=\infty$ when $\diffusion=0$) \cite[Lemma~3.2]{kkr}, (c) $\lim_{z\downarrow 0}\frac{W^{(q_1)}(z)}{W^{(q_2)}(z)}=1$ whenever $\{q_1,q_2\}\subset [0,\infty)$ (the latter follows from \cite[Eq.~(56)]{kkr}), and (d) the observation that $H(q)$ (see Remark~\ref{remark:conditioned} for the introduction of $H$) is non-vanishing, which we now verify.

In case \ref{conditioned:i}, $H(q)$ is clearly non-zero, it being equal to $\Phi(q)>0$. 

In case \ref{conditioned:ii}, one notes that by \cite[Theorem~8.1(iii)]{kyprianou}, for $x\in [-a,a]$,  $h_q(x)\geq \PP^x(\tau_{-a}^-\land \tau_a^+>e_q)=1-\PP^x[e^{-q(\tau_{-a}^-\land\tau_a^+)}]=q\Iq(2a)\frac{\Wq(x+a)}{\Wq(2a)}-q\Iq(x+a)$, where
 $\Iq(z):=\int_0^z\Wq(y)dy$ for $z\in \mathbb{R}$; 
hence from the fundamental theorem of calculus and from $(\Wq)'(0+)=2/\diffusion$, one sees that  $\lim_{x\downarrow -a}\frac{h_q(x)}{\Wq(x+a)}\geq \frac{q\Iq(2a)}{\Wq(2a)}-q\frac{ \diffusion}{2}\Wq(0)=\frac{q\Iq(2a)}{\Wq(2a)}>0$. 

By the same token, for case \ref{conditioned:iii}, $\lim_{x\uparrow a}\frac{h_q(x)}{a-x}\geq \lim_{x\uparrow a}\frac{\PP^x(\tau_{-a}^-\land \tau_a^+>e_q)}{a-x}=\frac{q}{\Wq(2a)}\left(\Wq(2a)^2-\Iq(2a)(\Wq)'(2a)\right)$, where we note that the latter expression is nondecreasing in $a$. Suppose for some given $a=a_0/2$ and $q$ it vanishes. Then for the given $q$ it vanishes for all $a\in (0,a_0/2)$. Hence on $(0,a_0)$ we would have identically $(\Wq)^2=(\Wq)'\int_0^\cdot  \Wq$, viz. $(\ln \Wq)'=(\ln(\int_0^\cdot \Wq))'$, and so by the fundamental theorem of calculus, for some $C\in (0,\infty)$,  $\Wq=C \int_0^\cdot \Wq$, which easily leads to a contradiction. 

Finally, for case \ref{conditioned:iv}, one has from \cite[Theorem~8.1(ii)]{kyprianou} that for $x\in [a,\infty)$, $h_q(x)\geq \PP^x(\tau_a^->e_q)=\frac{q}{\Phi(q)}\Wq(x-a)-q\Iq(x-a)$; hence $\lim_{x\downarrow a}\frac{h_q(x)}{\Wq(x-a)} \geq \frac{q}{\Phi(q)}-\frac{\diffusion}{2}q\Wq(0)=\frac{q}{\Phi(q)}>0$. 
\end{proof}

\section{Laplace transforms of entrance laws}\label{section:entrance-laws}
We will need the following excursion-theoretic lemma.
\begin{lemma}\label{lemma}
Let $q\in (0,\infty)$. Then: 
\begin{enumerate}[(i)]
\item\label{lemma:i} $\hat \PP^a[ 1-e^{-q \rho}]=\frac{\alpha_a(q)}{\alpha_a(1)}\frac{\Phi'(1)}{\Phi'(q)}=\frac{e^{\Phi(q)2a}}{e^{\Phi(1)2a}}\frac{\Phi'(q)^{-1}-\frac{e^{\Phi(q)2a}-1}{\Wq(2a)}}{\Phi'(1)^{-1}-\frac{e^{\Phi(1)2a}-1}{W^{(1)}(2a)}}$ and  $\hat \PP^{-a}[1-e^{-q \rho}]=\frac{\alpha_{-a}(q)}{\alpha_{-1}(1)}\frac{\Phi'(1)}{\Phi'(q)}=\frac{(e^{\Phi(q)2a}-1)}{(e^{\Phi(1)2a}-1)}\frac{W^{(1)}(2a)}{\Wq(2a)}$.

\item \label{lemma:iii} $\hat \PP^{-a}[1-e^{-q\rho};\tau_{-a}^-=0]=\frac{W^{(1)}(2a)}{e^{\Phi(1)2a}-1}\frac{\diffusion}{2}\Phi(q)$.
\item \label{lemma:iv} $\hat{\PP}^{-a}[1-e^{-q\rho};\tau_{-a}^+=0]=\frac{W^{(1)}(2a)}{e^{\Phi(1)2a}-1}\left(\frac{e^{\Phi(q)2a}-1}{\Wq(2a)}-\frac{\diffusion}{2}\Phi(q)\right)$.
\end{enumerate}
\end{lemma}
\begin{remark}\label{remark:lemma}
Recall that when $\diffusion=0$ then $\hat \PP^{a}(\tau_a^-=0)=0$, and it is clear from the above that $\hat\PP^{-a}(\tau_{-a}^-=0)=0$. Missing, therefore, in this lemma are (tractable) expressions for $\hat \PP^{a}[1-e^{-q\rho};\tau_{ a}^\pm=0]$, when $\diffusion>0$,  which it seems are more difficult to procure (cf. Lemma~\ref{remark:missing}).
\end{remark}
\begin{proof}
\ref{lemma:i}. Taking $Z=e^{-q \cdot}$ and $f=1-e^{-q \rho}$ in the master formula, we obtain thanks to the regularity of $a$ and $-a$ for $V$, and  since the potential measures of $X$ are absolutely continuous with respect to Lebesgue measure with $V$ being a Lebesgue measure zero set,
\footnotesize $$1=\PP^{\pm a}\left[\int_0^\infty qe^{-qs}\mathbbm{1}_{\mathbb{R}\backslash V}(X_s)ds\right]=\alpha_a(1)\hat \PP^a[1-e^{-q \rho}]\PP^{\pm a}\left[\int_0^\infty e^{-q s}dL^a_s\right]+\alpha_{-a}(1)\hat \PP^{-a}[1-e^{-q \rho}]\PP^{\pm a}\left[\int_0^\infty e^{-q s}dL^{-a}_s\right],$$\normalsize  i.e.
$$1=\alpha_a(1)\hat \PP^a[1-e^{-q \rho}]\PP^{a}\left[\int_0^\infty e^{-q s}dL^a_s\right]+\alpha_{-a}(1)\hat \PP^{-a}[1-e^{-q \rho}]\PP^a[e^{-qT_{-a}}]\PP^{-a}\left[\int_0^\infty e^{-q s}dL^{-a}_s\right]$$ and 
$$1=\alpha_a(1)\hat \PP^a[1-e^{-q \rho}]\PP^{-a}[e^{-qT_a}]\PP^{a}\left[\int_0^\infty e^{-q s}dL^a_s\right]+\alpha_{-a}(1)\hat \PP^{-a}[1-e^{-q \rho}]\PP^{-a}\left[\int_0^\infty e^{-q s}dL^{-a}_s\right].$$

Now \cite[Lemma~V.3]{bertoin} $\PP^{\pm a}\left[\int_0^\infty e^{-qt}dL^{\pm a}_t\right]=
\frac{u_q(0)}{u_1(0)}=\frac{\Phi'(q)}{\Phi'(1)}$, and the formulae follow. 

\ref{lemma:iii}. We have, using \cite[Lemma~2(iv)]{ppr} in the last equality, $\hat \PP^{-a}[e_q<\rho,\tau_{-a}^-=0]=\hat\PP^{-a}[e_q<\tau_{-a}^+,\tau_{-a}^-=0]=\frac{1}{\alpha_{-a}(1)}\tilde{\PP}^{-a}[e_q<T_{-a},\tau_{-a}^-=0]=\frac{\Phi'(1)}{\alpha_{-a}(1)}\nn[e_q<T_0,\tau_0^-=0]=\frac{W^{(1)}(2a)}{e^{\Phi(1)2a}-1}\frac{\diffusion}{2}\Phi(q)$.

\ref{lemma:iv}.  Follows from \ref{lemma:i} and \ref{lemma:iii} because $\hat\PP^{-a}[\rho=0]=0$ and hence $\Omega=\{\tau_{-a}^+=0\}\cup \{\tau_{-a}^-=0\}$ a.e.-$\hat{\PP}^{-a}$.
\end{proof}
Our last main result identifies the Laplace transforms of the entrance laws (under certain assumptions on $\diffusion$).

\begin{theorem}\label{theorem:main}
For $\beta\in (0,\infty)$, $\hat\eta ^{-a}_\beta$, and also $\hat \eta^a_\beta$ when $\diffusion=0$, is absolutely continuous with respect to Lebesgue measure and Lebesgue-a.e. in $y\in \mathbb{R}$ we have 
\begin{enumerate}[(i)]
\item\label{theorem:main:i} $\frac{\hat \eta^{-a}_\beta(dy)}{dy}=\frac{W^{(1)}(2a)}{e^{\Phi(1)2a}-1}\left[ \frac{W^{(\beta)}(a-y)-W^{(\beta)}(-a-y)e^{\Phi(\beta)2a}}{ W^{(\beta)}(2a)}\right],$ while
\item\label{theorem:main:ii}  when $\diffusion=0$, then \footnotesize $$\frac{\hat \eta_\beta^{a}(dy)}{dy}=\frac{\Phi'(1)}{\alpha_a(1)}\left[e^{\Phi(\beta)(a-y)}-\frac{W^{(\beta)}(a-y)}{W^{(\beta)}(2a)}e^{\Phi(\beta)2a}-W^{(\beta)}(-a-y)e^{\Phi(\beta)2a}\left(\Phi'(\beta)^{-1}-\frac{e^{\Phi(\beta)2a}}{W^{(\beta)}(2a)}\right)\right].$$\normalsize
\end{enumerate}
 \end{theorem}
\begin{remark}\label{remark:missing}
The case $\hat \eta^a_\beta$ when $\diffusion>0$ is missing because in Lemma~\ref{lemma} we are missing the pertinent (tractable) expressions in this case (cf. Remark~\ref{remark:lemma}).
\end{remark}
\begin{remark}
In the proof, more precisely, the quantity $$\frac{\int_0^\infty e^{-\beta t}\hat \PP^{-a}(X_t\in dy,\tau_{-a}^-=0)dt}{dy}=\frac{\diffusion}{2}\frac{W^{(1)}(2a)}{e^{\Phi(1)2a}-1}\left((W^{(\beta)})'(-a-y)-\Phi(\beta)W^{(\beta)}(-a-y)\right)$$ (and hence also the quantity $\frac{\int_0^\infty e^{-\beta t}\hat \PP^{-a}(X_t\in dy,\tau_{-a}^+=0)dt}{dy}$) is identified. 
\end{remark}
\begin{proof}
Let $f:\mathbb{R}\to \mathbb{R}$ be a continuous bounded function. Then, for $t\in (0,\infty)$, by monotone convergence, 
$\hat \PP^{\pm a}[f(X_t);t<\rho]=\lim_{q\downarrow 0}\hat\PP^{\pm a}[f(X_t);t<\rho\land e_q]$, and further for each $q\in (0,\infty)$, by dominated convergence and the Markov property,
$$\hat \PP^{\pm a}[f(X_t);t <\rho\land e_q]=\lim_{s\downarrow 0}\hat\PP^{\pm a}[f(X_{t+s});t+s<\rho\land e_q]=\lim_{s\downarrow 0}\hat \PP^{\pm a}[\PP^{X_s}[f(X_t);t<\rho\land e_q];s<\rho\land e_q]$$ \footnotesize $$=\lim_{s\downarrow 0}\hat \PP^{\pm a}\left[\frac{\PP^{X_s}[f(X_t);t<\rho\land e_q]}{h_q(X_s)}h_q(X_s);s<\rho\land e_q\right]=\lim_{s\downarrow 0}\underbrace{\hat \PP^{\pm a}\left[\PP^{X_s}_{\times,q}\left[\frac{f(X_t)}{h_q(X_t)};t<\zeta\right]h_q(X_s);s<\rho\land e_q\right]}_{\leq \Vert f\Vert_\infty \hat{\PP}^{\pm a}(t<\rho)}.$$\normalsize
Thus, for $\beta\in (0,\infty)$, by dominated convergence  and Tonelli-Fubini (use the estimate noted in the preceding display and Lemma~\ref{lemma}\ref{lemma:i}), recalling the notation of Definition~\ref{definition}, one obtains $$\hat \eta^{\pm a}_\beta(f)=\lim_{q\downarrow 0}\lim_{s\downarrow 0}\hat \PP^{\pm a}[Z_{q,\beta}(X_s,f/h_q)h_q(X_s);s<\rho\land e_q].$$ 

\ref{theorem:main:i}. We compute the limits of $\hat\PP^{-a}[Z_{q,\beta}(X_s,f/h_q)h_q(X_s);s<\rho\land e_q,\tau_{- a}^{\pm}=0]$ separately, then \ref{theorem:main:i} is got by taking the sum. 

Consider now first $\lim_{q\downarrow 0}\lim_{s\downarrow 0}\hat\PP^{-a}[Z_{q,\beta}(X_s,f/h_q)h_q(X_s);s<\rho\land e_q,\tau_{- a}^-=0]$ for $f$ compactly supported in $\mathbb{R}\backslash \set$, which implies that $f/h_q$ is also continuous and compactly supported in $\mathbb{R}\backslash \set$. 

Let $q\in (0,\infty)$. By Theorem~\ref{theorem:conditioned} for every $\epsilon>0$ there is a $\delta>0$ such that for all $s\in (0,\infty)$, $\{0> X_s+a\geq -\delta\}\subset \{\vert Z_{q,\beta}(X_s,f/h_q)-Z_{q,\beta}^-(-a,f/h_q)\vert \leq \epsilon\}$. Consequently, for  $s\in (0,\infty)$, one has the estimate $$\hat\PP^{-a}[\vert Z_{q,\beta}(X_s,f/h_q)-Z_{q,\beta}^-(-a,f/h_q)\vert h_q(X_s);s<\rho\land e_q,\tau_{- a}^-=0]$$
$$\leq \epsilon \hat\PP^{-a}[h_q(X_s);s<\rho\land e_q,\tau_{- a}^-=0]+\hat\PP^{-a}[h_q(X_s)\mathbbm{1}_{\{\vert X_s+a\vert >\delta\}\cup \{X_s\geq -a\}},s<\rho\land e_q,\tau_{- a}^-=0]$$
$$=\epsilon \hat\PP^{-a}(s<e_q<\rho,\tau_{-a}^-=0)+\hat\PP^{-a}[h_q(X_s)\mathbbm{1}_{\{\vert X_s+a\vert >\delta\}\cup \{X_s\geq -a\}},s<\rho\land e_q,\tau_{- a}^-=0],$$ since
$\hat\PP^{-a}[h_q(X_s);s<\rho\land e_q,\tau_{- a}^-=0]=\hat\PP^{-a}(s<e_q<\rho,\tau_{-a}^-=0)$. For the same reason we conclude via dominated convergence (using Lemma~\ref{lemma}\ref{lemma:i} and the fact that $\hat\PP^{-a}(X_0\ne -a)=0$) for the second term, and say monotone convergence for the first term of the preceding display, that \footnotesize $$\limsup_{s\downarrow 0}\hat\PP^{-a}[\vert Z_{q,\beta}(X_s,f/h_q)-Z_{q,\beta}^-(-a,f/h_q)\vert h_q(X_s);s<\rho\land e_q,\tau_{- a}^-=0]\leq \epsilon \hat\PP^{-a}(e_q<\rho,\tau_{-a}^-=0)\to 0\text{ as }\epsilon\downarrow 0.$$\normalsize

By this it is proved that 
$$\lim_{q\downarrow 0}\lim_{s\downarrow 0}\hat\PP^{-a}[Z_{q,\beta}(X_s,f/h_q)h_q(X_s);s<\rho\land e_q,\tau_{- a}^-=0]$$ $$= \lim_{q\downarrow 0}Z^-_{q,\beta}(-a,f/h_q)\lim_{s\downarrow 0}\hat\PP^{-a}(s<e_q<\rho,\tau_{-a}^-=0)= \lim_{q\downarrow 0}Z^-_{q,\beta}(-a,f/h_q)\hat\PP^{-a}(e_q<\rho,\tau_{-a}^-=0),$$ which by Theorem~\ref{theorem:conditioned}\ref{conditioned:i} and  Lemma~\ref{lemma}\ref{lemma:iii} 
$$= \lim_{q\downarrow 0}\int f(y)\frac{(W^{(\beta+q)})'(-a-y)-\Phi(\beta+q)W^{(\beta+q)}(-a-y)}{\Phi(q)}dy\frac{W^{(1)}(2a)}{e^{\Phi(1)2a}-1}\frac{\diffusion}{2}\Phi(q).$$
$$= \lim_{q\downarrow 0}\int f(y)\left((W^{(\beta+q)})'(-a-y)-\Phi(\beta+q)W^{(\beta+q)}(-a-y)\right)dy\frac{W^{(1)}(2a)}{e^{\Phi(1)2a}-1}\frac{\diffusion}{2}$$
that finally converges to $ \int f(y)((W^{(\beta)})'(-a-y)-\Phi(\beta)W^{(\beta)}(-a-y))dy\frac{W^{(1)}(2a)}{e^{\Phi(1)2a}-1}\frac{\diffusion}{2}$ as $q\downarrow 0$. (The continuity of $(\Wq)'$ in $q\in (0,\infty)$ (indeed analyticity of $(\Wq)'$ in $q\in \mathbb{C}$) follows from \cite[Eq.~(8.29)]{kyprianou}.) 

Similarly, again for $f$ compactly supported in $\mathbb{R}\backslash \set$, $\lim_{q\downarrow 0}\lim_{s\downarrow 0}\hat\PP^{-a}[Z_{q,\beta}(X_s,f/h_q)h_q(X_s);s<\rho\land e_q,\tau_{- a}^+=0]=\lim_{q\downarrow 0}Z^+_{q,\beta}(-a,f/h_q)\hat\PP^{-a}(e_q<\rho,\tau_{-a}^-=0)$, which by Theorem~\ref{theorem:conditioned}\ref{conditioned:ii} and  Lemma~\ref{lemma}\ref{lemma:iv} is equal to $ \int f(y)\left[ \frac{W^{(\beta)}(a-y)-W^{(\beta)}(-a-y)e^{\Phi(\beta)2a}}{ W^{(\beta)}(2a)} -\frac{\diffusion}{2}((W^{(\beta)})'(-a-y)-\Phi(\beta)W^{(\beta)}(-a-y))\right]dy\frac{W^{(1)}(2a)}{e^{\Phi(1)2a}-1}$.

\ref{theorem:main:ii} follows by near identical lines of reasoning using Theorem~\ref{theorem:conditioned}\ref{conditioned:iv} and Lemma~\ref{lemma}\ref{lemma:i}, since thanks to $\diffusion=0$, $\hat\PP^a$ is carried by $\{\tau_a^+=0\}$; we omit the details.
\end{proof}
\section{An application}\label{section:applications}
For $\{x,y\}\subset \mathbb{R}$ define $$S_{x,y}:=\sup\{t\in (0,T_y):X_t=x\}\quad (\sup\emptyset=0),$$
the last time the process is at $x$ before it hits $y$ ($=0$, when there is no such time). We are interested in the Laplace transform of $T_y-S_{x,y}$ on $\{T_y<\infty\}$, viz. of the time that elapses between the last visit to $x$ before hitting $y$ and the hitting time of $y$, on the event that $y$ is hit at all.

\begin{lemma}\label{lemma:to-laplace}
Let $\lambda\in [0,\infty)$, $x\in \mathbb{R}$. Then 
\begin{enumerate}[(i)]
\item \label{lemma:to-laplace:ii} $\PP^x[e^{-\lambda \rho};T_{a}<T_{-a}]=\frac{\Wl(x+a)-\Wl(x-a)e^{\Phi(\lambda)2a}}{\Wl(2a)}$; in particular when $x\in (-\infty,a]$, then $\PP^x[e^{-\lambda\rho};T_{a}<T_{-a}]=\frac{\Wl(x+a)}{\Wl(2a)}$; and if even $x\in (-\infty,-a]$, then (of course) $\PP^x[e^{-\lambda\rho};T_{a}<T_{-a}]=0$.
\item\label{lemma:to-laplace:i} $\PP^x[e^{-\lambda \rho};T_{-a}<T_a]=e^{\Phi(\lambda)(x+a)}-\frac{\Wl(x+a)}{\Wl(2a)}e^{\Phi(\lambda)2a}-e^{\Phi(\lambda)2a}\Wl(x-a)(\Phi'(\lambda)^{-1}-\frac{e^{\Phi(\lambda)2a}}{\Wl(2a)})$; in particular when $x\in (-\infty,a]$, then $\PP^x[e^{-\lambda\rho};T_{-a}<T_a]=e^{\Phi(\lambda)(x+a)}-\frac{\Wl(x+a)}{\Wl(2a)}e^{\Phi(\lambda)2a}$; and if even $x\in (-\infty,-a]$, then $\PP^x[e^{-\lambda\rho};T_{-a}<T_a]=e^{\Phi(\lambda)(a+x)}$.
\end{enumerate}
\end{lemma}
\begin{remark}
When $\psi'(0+)=0=\lambda$, $\Phi'(\lambda)^{-1}$ is interpreted in the limiting sense, i.e. it is then equal to $0$.
\end{remark}
\begin{proof}
\ref{lemma:to-laplace:ii}. Since a.s. $X$ has no positive jumps it is clear that $\PP^x[e^{-\lambda\rho};T_{a}<T_{-a}]=0$ for $x\in (-\infty,-a]$. Let $x\in (-a,a]$. Then $\PP^x[e^{-\lambda\rho};T_{a}<T_{-a}]=\PP^x[e^{-\lambda\tau_a^+};\tau_a^+<\tau_{-a}^-]=\frac{\Wl(x+a)}{\Wl(2a)}$. Finally for $x\in (a,\infty)$, one has $\PP^x[e^{-\lambda\rho};T_{a}<T_{-a}]=\PP^x\left[e^{-\lambda \tau_a^-}\PP_{X_{\tau_a^-}}[e^{-\lambda \rho};T_{-a}<T_a];\tau_a^-<\infty\right]=\PP^x\left[e^{-\lambda \tau_a^-}\frac{\Wl(X_{\tau_a^-}+a)}{\Wl(2a)};\tau_a^-<\infty\right]$, which thanks to \eqref{eq:help:two:ii} is $=\frac{\Wl(x+a)-\Wl(x-a)e^{\Phi(\lambda)2a}}{\Wl(2a)}$.

\ref{lemma:to-laplace:i}. For $x\in (-\infty,-a]$, one has $\PP^x[e^{-\lambda\rho};T_{-a}<T_a]=\PP^x[e^{-\lambda \tau_{-a}^+};\tau_{-a}^+<\infty]=e^{\Phi(\lambda)(a+x)}$. 
If $x\in (-a,a]$, then using \eqref{eq:help:one:i} one obtains $\PP^x[e^{-\lambda\rho};T_{-a}<T_a]=\PP^x[\PP_{X_{\tau_{-a}^-}}[e^{-\lambda \tau_{-a}^+};\tau_{-a}^+<\infty]e^{-\lambda\tau_{-a}^-};\tau_{-a}^-<\tau_a^+]=\PP^x[e^{\Phi(\lambda)(X_{\tau_{-a}^-}+a)-\lambda\tau_{-a}^-};\tau_{-a}^-<\tau_a^+]=e^{\Phi(\lambda)(x+a)}-\frac{\Wl(x+a)}{\Wl(2a)}e^{\Phi(\lambda)2a}$. Finally if $x\in  (a,\infty)$, then using \eqref{eq:help:one:ii} and \eqref{eq:help:two:ii}, $\PP^x[e^{-\lambda \rho};T_{-a}<T_a]=\PP^x\left[e^{-\lambda \tau_a^-}\PP_{X_{\tau_a^-}}[e^{-\lambda \rho};T_{-a}<T_a];\tau_a^-<\infty\right]=\PP^x\left[e^{-\lambda \tau_a^-}\left(e^{\Phi(\lambda)(X_{\tau_a^-}+a)}-\frac{\Wl(X_{\tau_a^-}+a)}{\Wl(2a)}e^{\Phi(\lambda)2a}\right);\tau_a^-<\infty\right]=e^{\Phi(\lambda)2a}(e^{\Phi(\lambda)(x-a)}-\Wl(x-a)\Phi'(\lambda)^{-1})-\frac{e^{\Phi(\lambda)2a}}{\Wl(2a)}(\Wl(x+a)-\Wl(x-a)e^{\Phi(\lambda)2a})$.
\end{proof}

\begin{theorem}\label{theorem:application}
Let $x<y$ and $z$ be real numbers. Then for $\lambda\in [0,\infty)$,
$$\PP^z[e^{-\lambda (T_y-S_{x,y})};T_y<\infty]=\frac{\Wl(z-x)-\Wl(z-y)e^{\Phi(\lambda)(y-x)}}{\Wl(y-x)}+\frac{W(y-x)}{\Wl(y-x)}e^{-\Phi(0)(y-x)}$$
$$\times\left[e^{\Phi(\lambda)(z-x)}-\frac{\Wl(z-x)}{\Wl(y-x)}e^{\Phi(\lambda)(y-x)}-e^{\Phi(\lambda)(y-x)}\Wl(z-y)\left(\Phi'(\lambda)^{-1}-\frac{e^{\Phi(\lambda)(y-x)}}{\Wl(y-x)}\right)\right].$$
\end{theorem}
\begin{remark}
In the proof it will be seen that $\PP^x[L^x_{T_y}]=(\Phi'(1))^{-1}W(y-x)e^{-\Phi(0)(y-x)}$.
\end{remark}
\begin{proof}
By the master formula, taking $Z_s=\mathbbm{1}_{\{s<T_{a}\}}$ and $f=e^{-\lambda \rho}\mathbbm{1}_{\{T_{a}<T_{-a}\}}$ therein, we have
$$\PP^{-a}[e^{-\lambda (T_{a}-S_{- a, a})};T_{a}<\infty]=\alpha_{-a}(1)\PP^{-a}[ L^{-a}_{T_{ a}}]\hat{\PP}^{ -a}[e^{-\lambda \rho};T_{a}<T_{-a}],$$ and 
$$\hat{\PP}^{-a}[e^{-\lambda \rho};T_{a}<T_{-a}]=\lim_{t\downarrow 0}\hat{\PP}^{-a}[e^{-\lambda \rho};T_{a}<T_{-a},t<\rho]=\lim_{t\downarrow 0}\hat{\PP}^{-a}[\PP^{X_t}[e^{-\lambda\rho};T_{a}<T_{-a}];t<\rho]$$
$$=\lim_{\beta\to \infty}\beta \int_0^\infty e^{-\beta t}\hat{\PP}^{-a}[\PP^{X_t}[e^{-\lambda\rho};T_{a}<T_{-a}];t<\rho]dt=\lim_{\beta\to \infty}\beta\hat{\eta}_\beta^{-a}[\PP^\cdot[e^{-\lambda\rho};T_{a}<T_{-a}]].$$

We have from Lemma~\ref{lemma:to-laplace}\ref{lemma:to-laplace:ii} and Theorem~\ref{theorem:main}\ref{theorem:main:i} that $\beta\hat{\eta}_\beta^{-a}[\PP^\cdot[e^{-\lambda\rho};T_{a}<T_{-a}]]$ is equal to 
$$\beta \int_{-a}^a\frac{W^{(1)}(2a)}{e^{\Phi(1)2a}-1}\frac{W^{(\beta)}(a-y)}{ W^{(\beta)}(2a)}\frac{\Wl(y+a)}{\Wl(2a)}dy=\beta \int_{0}^{2a}\frac{W^{(1)}(2a)}{e^{\Phi(1)2a}-1}\frac{W^{(\beta)}(2a-y)}{ W^{(\beta)}(2a)}\frac{\Wl(y)}{\Wl(2a)}dy$$\footnotesize
$$=\frac{W^{(1)}(2a)}{e^{\Phi(1)2a}-1}\frac{\beta}{W^{(\beta)}(2a)\Wl(2a)}(W^{(\beta)}\star \Wl)(2a)=\frac{W^{(1)}(2a)}{e^{\Phi(1)2a}-1}\frac{\beta}{W^{(\beta)}(2a)\Wl(2a)}\frac{W^{(\beta)}(2a)-\Wl(2a)}{\beta-\lambda}$$\normalsize (because $(\beta-\lambda)W^{(\beta)}\star \Wl=W^{(\beta)}-\Wl$ as follows for instance by taking Laplace transforms), which converges to $\frac{W^{(1)}(2a)}{(e^{\Phi(1)2a}-1)\Wl(2a)}$ as $\beta\to\infty$. In consequence 
$$\PP^{- a}[e^{-\lambda (T_{ a}-S_{-a, a})};T_{a}<\infty]=\alpha_{- a}(1)\PP^{ - a}[ L^{- a}_{T_{ a}}]\frac{W^{(1)}(2a)}{(e^{\Phi(1)2a}-1)\Wl(2a)}=\PP^{ - a}[ L^{- a}_{T_{ a}}]\frac{\Phi'(1)}{\Wl(2a)}.$$
Taking $\lambda=0$ yields $\PP^{ - a}[ L^{- a}_{T_{ a}}]\frac{\Phi'(1)}{W(2a)}=\PP^{- a}(T_{a}<\infty)=\PP^{-a}(\tau_a^+<\infty)=e^{-\Phi(0)2a}$, so that 
$$\PP^{- a}[e^{-\lambda (T_{ a}-S_{-a, a})};T_{a}<\infty]=\frac{W(2a)}{\Wl(2a)}e^{-\Phi(0)2a}.$$

More generally $\PP^z[e^{-\lambda (T_y-S_{x,y})};T_y<\infty]=\PP^z[e^{-\lambda T_y};T_y<T_x]+\PP^z(T_x<T_y)\PP^x[e^{-\lambda (T_y-S_{x,y})};T_y<\infty]$ and Lemma~\ref{lemma:to-laplace} allows to conclude. 
\end{proof}

\begin{remark}
It is possible to consider the case $y<x$, at least when $\diffusion=0$, by using Lemma~\ref{lemma:to-laplace}\ref{lemma:to-laplace:i} and Theorem~\ref{theorem:main}\ref{theorem:main:ii} together with the same technique as in the proof above; for sure an integral representation for the Laplace transform involving the scale functions can be produced. However the resulting integral is much more involved in this case and it was not possible to produce a simple expression for this constellation of $x$ and $y$.
\end{remark}

\bibliographystyle{plain}
\bibliography{Excursions_of_SNLP_from_set}
\end{document}